\documentclass[12pt]{amsart}
\usepackage{amsmath}
\usepackage{amsfonts}
\usepackage{amssymb}
\usepackage{graphicx}%
\providecommand{\U}[1]{\protect\rule{.1in}{.1in}}
\newtheorem {theorem}{Theorem}

\usepackage{graphicx}%
\providecommand{\U}[1]{\protect\rule{.1in}{.1in}}

\newtheorem {definition}[theorem]{Definition}

\newtheorem {remark}[theorem]{Remark}

\newtheorem {rmk}{Remark}

\begin{document}
\title{\textbf{Some Observations About ``Moving Norms'' in Normed Algebras}}
\author{Eliahu Levy}
\address{Eliahu Levy\\Department of Mathematics\\Technion 
-- Israel Institute of Technology\\Technion City, Haifa 3200003, Israel}
\email{eliahu@math.technion.ac.il}

\date{}
\thanks{I am much indebted and grateful to Professor Yair Censor of the Haifa University, this note stemming from joint work with him, and to Professors Yoav Benyamini, Simeon Reich and Orr Shalit of the Technion, Haifa for their constant kind advice.}
 
\begin{abstract}
In this note one tries to venture into a study of some notions, in the context of a (unital) normed algebra, in particular the algebra of operators on a Hilbert space. Namely, one considers ``moving norms'', i.e.\ norming an element minus a scalar multiple of the unity and related notions, and goes into obtaining some rather nice assertions.
\end{abstract}

\maketitle

\tableofcontents

\subsection{The Moving Norm and Augmented Moving Norm}
Recall, that a \texttt{normed (unital) algerbra} is a (unital) algebra $\mathcal{A}$ over
$\mathbb{R}$ or $\mathbb{C}$ as scalars, which is a normed vector space satisfying
\begin{eqnarray}
&&\Vert xy\Vert\leq\Vert x\Vert\,\Vert y\Vert\,\,\text{ for any }x,y\in\mathcal{A}, \\
&& \Vert \mathbf{1}\Vert=1,\text{ where }\mathbf{1}\text{ is the unit element}.
\end{eqnarray}
\begin{definition}
The \texttt{moving norm} of an $x\in\mathcal{A}$ is the (continuous) function on $\lambda\in\mathbb{R}^+$
given by
\begin{equation}
\Vert x\Vert_{m}(\lambda):=\Vert x-\lambda\mathbf{1}\Vert,\quad \lambda\geq 0.
\end{equation}
\end{definition}

Then one has

\begin{eqnarray}
&&\text{For }x\in\mathcal{A}\text{ and }c \geq 0,\quad
\Vert cx\Vert_{m}(c\lambda)=c\Vert x\Vert_{m}(\lambda);\label{eq:m-c} \\
&&\text{For }x,y\in\mathcal{A}\text{ and }\lambda,\mu\geq 0,\quad
\Vert x+y\Vert_{m}(\lambda+\mu) \leq
\Vert x\Vert_{m}(\lambda)+\Vert y\Vert_{m}(\mu); \label{eq:m-sum}\\
&&\text{Consequently, for }x\in\mathcal{A}\text{ and }\lambda\geq 0 \nonumber \\
&&\Vert x\Vert_{m}(\lambda) \leq 
\inf_{y+z=x,\,\mu+\nu=\lambda,\mu,\nu\geq 0}
\Vert y\Vert_{m}(\mu)+\Vert z\Vert_{m}(\nu).\label{eq:m-inf}
\end{eqnarray}

Define also

\begin{definition}
The \texttt{augmented moving norm} of an $x\in\mathcal{A}$ is the (continuous) function on $\lambda\in\mathbb{R}^+$
given by
\begin{equation}
\Vert x\Vert_{am}(\lambda):=\Vert x-\lambda\mathbf{1}\Vert+\lambda=\Vert x\Vert_a(\lambda)+\lambda,
\quad \lambda\geq 0.
\end{equation}
\end{definition}

This definition is motivated by some somewhat nice properties obtained (Thm.\ \ref{th:am}).

\begin{theorem}\label{th:am}
\begin{eqnarray}
&&\Vert cx\Vert_{am}(c\lambda)=c\Vert x\Vert_{am}(\lambda), 
\text{For }x,y\in\mathcal{A}\text{ and }\lambda\geq 0\label{eq:am-c}\\
&&\Vert x+y\Vert_{am}(\lambda+\mu) \leq
\Vert x\Vert_{am}(\lambda)+\Vert y\Vert_{am}(\mu), \label{eq:am-sum} \\
&&\text{Consequently, for }x\in\mathcal{A}\text{ and }\lambda\geq 0 \nonumber \\
&&\Vert x\Vert_{am}(\lambda) \leq 
\inf_{y+z=x,\,\mu+\nu=\lambda,\mu,\nu\geq 0}
\Vert y\Vert_{am}(\mu)+\Vert z\Vert_{am}(\nu).\label{eq:am-inf} \\
&&\text{For }x,y\in\mathcal{A}\text{ and }\lambda,\mu\geq 0\nonumber \\
&&\Vert x+y\Vert_{am}(\lambda+\mu) \leq
\Vert x\Vert_{am}(\lambda)+\Vert y\Vert_{am}(\mu).\\
&&\Vert xy\Vert_{am}(\lambda\mu) \leq
\Vert x\Vert_{am}(\lambda)\cdot\Vert y\Vert_{am}(\mu).\label{eq:am-p}
\end{eqnarray}
\end{theorem}

\begin{proof}
The first assertions about multiplying by a scalar and about a sum follow from
(\ref{eq:m-c}) and (\ref{eq:m-sum}).

As for the assertion (\ref{eq:am-p}) about a product, we have
\begin{eqnarray}
&&\text{For }x,y\in\mathcal{A}\text{ and }\lambda,\mu\geq 0\nonumber \\
&&\Vert x\Vert_{m}(\lambda) = \Vert x-\lambda\mathbf{1}\Vert,\,
\Vert y\Vert_{m}(\mu) = \Vert y-\mu\mathbf{1}\Vert,\nonumber \\
&&\text{makeing}\nonumber \\
&&\Vert xy-\lambda y-\mu x +\lambda\mu\mathbf{1}\Vert\leq
\Vert x\Vert_{m}(\lambda)\Vert y\Vert  _{m}(\mu),\nonumber \\
&&\text{thus}\nonumber \\
&& \Vert xy-\lambda\mu\mathbf{1}\Vert\leq
\Vert x\Vert_{m}(\lambda)\Vert y\Vert_{m}(\mu)+\Vert \lambda y-\lambda\mu\mathbf{1}\Vert+
\Vert \mu x-\lambda\mu\mathbf{1}\Vert=\nonumber \\
&& \Vert x\Vert_{m}(\lambda)\Vert y\Vert_{m}(\mu)+
\lambda \Vert y\Vert_{m}(\mu)+\mu\Vert x \Vert_{m}(\lambda).\nonumber \\
&&\Vert xy\Vert_{m}(\lambda\mu)\leq
\Vert x\Vert_{m}(\lambda)\Vert y\Vert_{m}(\mu)+
\lambda \Vert y\Vert_{m}(\mu)+\mu\Vert x \Vert_{m}(\lambda),\nonumber \\
&&\Vert xy\Vert_{am}(\lambda\mu)=\Vert xy\Vert_{m}(\lambda\mu)+\lambda\mu \leq
\Vert x\Vert_{m}(\lambda)\Vert y\Vert_{m}(\mu)+
\lambda \Vert y\Vert_{m}(\mu)+\mu\Vert x \Vert_{m}(\lambda)+\lambda\mu=\nonumber \\
&&\Big(\Vert x\Vert_{m}(\lambda)+\lambda\Big)\cdot\Big(\Vert y\Vert_{m}(\mu)+\mu\Big)=
\Big(\Vert x\Vert_{am}(\lambda)\Big)\cdot\Big(\Vert y\Vert_{am}(\mu)\Big)\nonumber
\end{eqnarray}
Proving (\ref{eq:am-p}).
\end{proof}

\subsection{The Horizon of an Element of Norm $\leq 1$}
Let $x\in\mathcal{A}$ be of norm $\Vert x \Vert\leq 1$. Its augmented moving norm, a continuous function $\mathbb{R}^\to+\mathbb{R}^+$\,\,$\lambda\mapsto\Vert x\Vert_{am}(\lambda)$ is convex (by Thm.\ \ref{th:am}), taking at $0$ the value $\Vert x\Vert\leq 1$ and is $\geq\lambda$, so it turns to $\infty$ as $\lambda\to\infty$. Therefore there is a unique $\lambda_0 \geq 0$ so that at $\lambda_0$ it takes the value $1$ and has value $>1$ at greater places.

\begin{definition}\label{def:Hor}
For an $x\in\mathcal{A}$ of norm $\Vert x \Vert\leq 1$, the \texttt{horizon} of $x$ is defined as
the unique nonnegative number $\operatorname{Hor}(x)$ so that
\begin{equation}
\Vert x\Vert_{am}(\operatorname{Hor}(x))=1,\quad
\lambda>\operatorname{Hor}(x)\Rightarrow\Vert x\Vert_{am}(\lambda)>1.
\end{equation}
\end{definition}

\begin{remark}
Then by the convexity of $\lambda\mapsto\Vert x\Vert_{am}(\lambda)$,  
\begin{eqnarray}
&&\text{either }\Big(\lambda<\operatorname{Hor}(x)\,\Rightarrow\,\Vert x\Vert_{am}(\lambda)<1\Big),
\text{ hence if }\operatorname{Hor}(x)>0\text{ then }\Vert x \Vert=\Vert x \Vert_{am}(0)<1, \\
&&\text{or }\Big(0\leq\lambda \leq \operatorname{Hor}(x)\,\Rightarrow\,\Vert x\Vert_{am}(\lambda)=1\Big),
\text{ in particular }\Vert x \Vert=\Vert x \Vert_{am}(0)=1.
\end{eqnarray}

And clearly, as we always have $\Vert x\Vert_{am}(\lambda) \geq \lambda$,
\begin{equation}
\operatorname{Hor}(x)\leq 1\text{ for every }x.
\end{equation}
\end{remark}

\begin{theorem}\label{th:Hor}
For $x,y\in\mathcal{A}$ with norms $\Vert x\Vert, \Vert y\Vert \leq 1$
\begin{eqnarray}
&&\operatorname{Hor}\Big(tx+(1-t)y\Big)
\geq t\operatorname{Hor}(x)+(1-t)\operatorname{Hor}(y)
\quad 0\leq t \leq 1. \label{eq:Hor-s}\\
&&\operatorname{Hor}(xy) \geq \operatorname{Hor}(x)\cdot\operatorname{Hor}(y).
\label{eq:Hor-p} 
\end{eqnarray}
\end{theorem}
\begin{proof}
By Def.\ \ref{def:Hor} and (\ref{eq:am-c}) we have that  

$\Vert tx\Vert_{am}(t\lambda)=t\Vert x\Vert_{am}(\lambda)$
is $\leq t$ at $\lambda=\operatorname{Hor}(x)$.

Similarly $\Vert (1-t)y\Vert_{am}((1-t)\lambda)=(1-t)\Vert y\Vert_{am}(\lambda)$
is $\leq 1-t$ at $\lambda=\operatorname{Hor}(y)$.

Therefore, by Thm.\ \ref{th:am},
$$\Vert tx+(1-t)y\Vert_{am}\Big(t\operatorname{Hor}(x)+(1-t)\operatorname{Hor}(y)\Big)\leq 1,$$
which implies (\ref{eq:Hor-s}).

Now, by Def.\ \ref{def:Hor} and Thm.\ \ref{th:am}, we have that $\Vert xy \Vert_{am}(\lambda)$
is $\leq 1$ at $\lambda=\operatorname{Hor}(x)\cdot\operatorname{Hor}(y)$,
implying (\ref{eq:Hor-p}).
\end{proof}

\subsection{Re Elements in a $C^\ast$-Algebra}

Consider a $C^\ast$-algebra $\mathcal{C}$.
 
Clealy for an $x\in\mathcal{C}$, the moving norm, augmented moving norm and horizon are the same for $x$ and $x^\ast$.

\begin{theorem}\label{th:x-xast}
For $x\in\mathcal{C}$,
\begin{eqnarray}
&&\Vert x x^\ast\Vert_{am}(\lambda^2),\,\Vert x^\ast x\Vert_{am}(\lambda^2) 
\leq \Big(\Vert x\Vert_{am}(\lambda)\Big)^2, \label{eq:amam}\\
&&\text{and when }\Vert x\Vert\leq 1,\quad\operatorname{Hor}(x x^\ast),\,\operatorname{Hor}(x^\ast x)
\geq \Big(\operatorname{Hor}(x)\Big)^2. \label{eq:HorHor}
\end{eqnarray}
\end{theorem}
\begin{proof}
Follows from Thm.\ \ref{th:am} and Thm.\ \ref{th:Hor}, 
\end{proof}

\begin{theorem}\label{th:uni}
Let $u\in\mathcal{C}$ be \textit{unitary}, i.e.\ $u^\ast=u^{-1}$, consequently $\Vert u\Vert=\sqrt{\Vert u^\ast u\Vert}=\sqrt{\Vert\mathbf{1}\Vert}=1$. Then
\begin{equation}
\Vert u\Vert_{am}(\lambda)\geq 1 \text{ for any }\lambda \geq 1 
\end{equation}
Thus
\begin{equation}
\Vert u\Vert_{am}(\lambda)=1 \text{ on }\lambda\in[0,\operatorname{Hor}(u)].
\end{equation} 
\end{theorem}
\begin{proof}
\begin{eqnarray}
&&\Vert u\Vert_m(\lambda)=\Vert u-\lambda\mathbf{1}\Vert=
\sqrt{\Vert(u^\ast-\lambda\mathbf{1})(u-\lambda\mathbf{1})\Vert}=\nonumber \\
&&\sqrt{\Vert(\lambda^2+1)\mathbf{1}-\lambda(u+u^\ast)\Vert}\geq
\sqrt{\lambda^2+1-2\lambda}=|1-\lambda|.\nonumber \\
&&\Vert u\Vert_{am}(\lambda)=\Vert u\Vert_{m}(\lambda)+\lambda \geq 
|1-\lambda|+\lambda\geq 1-\lambda+\lambda=1\nonumber
\end{eqnarray}
\end{proof}

\subsection{Re Operators on Hilbert Space}
Focus on the $C^\ast$-algebra of \texttt{bounded linear operators} on a Hilbert Space $\mathcal{H}$.

As is well known, The operator $A$ is called \texttt{nonexpansive} NE if
\begin{equation}
\Vert A(x)-A(y)\Vert \leq \Vert x-y\Vert ,{\text{\ for all\ }}x,y.
\end{equation}
Namely, $A$ being linear, if $\Vert A \Vert \leq 1$.
 
And $A$ is called \texttt{monotone} if
\begin{equation}
\langle y-x,Ay-Ax\rangle\geq0,{\text{\ for all\ }}x,y.
\end{equation}

These definitions describe the action of $A$ on a pair $x,y$ compared with the
original pair: Nonexpansive operators do not make the pair \textquotedblleft
further apart\textquotedblright, while monotone operators \textquotedblleft do
not rotate it in more than $90$ degrees.\textquotedblright\ A linear
orthogonal projection is NE and monotone, any linear operator with
norm $\leq1$ is NE while any linear operator whose symmetric part is
positive definite is monotone.

Now, as is well known, for a \textit{Hermitian} operator $A$, its \textit{spectral measure} (see \cite{DS-book}) is supported on $\mathbb{R}$, indeed on some bounded subset of $\mathbb{R}$. And the norm $\Vert A \Vert$
is the smallest $R$ so that the spectral measure of $A$ is supported inside the interval $[-R,R]$.

In other words, if $[a,b]$ is the convex hull of the support of the spectral measure then
$$\Vert A\Vert=\max(|a|,|b|).$$

Consequently $A$ is \textit{NE} if and only if its spectral measure is supported inside
the interval $[-1,1]$.

Now,
$$\Vert A \Vert_{am}(\lambda)=\Vert x-\lambda\mathbf{1}\Vert+\lambda=
\max(|a-\lambda|,|b-\lambda|)+\lambda.$$

If $\lambda \leq b$, that is equal to
$$=\max(a-\lambda,b-\lambda)+\lambda=a-\lambda+\lambda=a.$$
 
xut if $\lambda \geq a$, it will be
$$=\max(\lambda-a,\lambda-b)+\lambda=\lambda-b+\lambda=2\lambda-b$$

And if $a\geq \lambda \geq b$, we will have
$$=\max(a-\lambda,\lambda-b)+\lambda
=\left\{\ \begin{array}{l}
a \text{ if }\lambda \leq \frac{a+b}2\\
\lambda-b+\lambda=2\lambda-b\text{ if }\lambda \geq \frac{a+b}2
\end{array}\right. .$$

Thus,
\begin{theorem}\label{th:Her}
Let $A$ be a Hermitian operator on a Hilbert space $\mathcal{H}$, and let $[a,b]$ be the convex hull of the support of the spectral measure of $A$. then
\begin{eqnarray}
&&\Vert A \Vert_{am}(\lambda)
=\left\{\ \begin{array}{l}
a\text{ if }\lambda \leq \frac{a+b}2\\
2\lambda-b\,\,\text{ if }\lambda \geq \frac{a+b}2
\end{array}\right. \\
&&\text{Hence, when }A\text{ NE, }\operatorname{Hor}(A)=(b+1)/2.
\end{eqnarray}
\end{theorem}\begin{proof}\end{proof}

Thus, for a Hermitian NE $A$ to have $\operatorname{Hor}(A)>0$, resp.\ $=0$, it is necessary and sufficient that $b>-1$, resp $b=-1$ that is, \textit{that the negative part of $A$ has norm $<1$, resp.\ norm $1$}.

And for $A$ to have $\operatorname{Hor}(A)=1$, it is necessary and sufficient that $b=1$, that is, \textit{that $A$ is a scalar multiple of $\mathbf{1}$}.

\begin{theorem}
Let $B$ be some NE operator. Then by Thm.\ \ref{th:x-xast},
\begin{eqnarray*}
&&\operatorname{Hor}(B)>0\,\Rightarrow\,
\operatorname{Hor}(B^\ast B),\,\operatorname{Hor}(B B^\ast)>0 \\
&&\Rightarrow\text{ the negative parts of }B^\ast B\text{ and }B B^\ast\text{ have norm }<1. \\
&&\text{While} \\
&&\operatorname{Hor}(B)=1\,\Rightarrow\,\operatorname{Hor}(B^\ast B)=1,
\text{ thus }B^\ast B\text{ is a scalar multiple of }\mathbf{1}, \\
&&\text{so }B\text{ is a scalar multiple of a unitary}.
\end{eqnarray*}
\end{theorem}\begin{proof}\end{proof}

\subsection{Firmly Nonexpansive Operators}
An operator $A$ on a Hilbert space $\mathcal{H}$ is called (\cite{GR-book} p.\ 41) \texttt{firmly Nonexpansive} FNE if $A$ and $2A-\mathbf{1}$ are both NE.

\begin{itemize}
\item
In particular, an orthogonal projection on a subspace $P$ is FNE as $2P-1$ is a reflection in the
subspace.
\item
For an operator $A$ to be FNE we need that for any $x\in\mathcal{H}$,
\begin{eqnarray}
&&\Vert Ax\Vert^2\leq \Vert x\Vert \nonumber \\
&&\Vert 2Ax-x\Vert^2\leq \Vert x\Vert \nonumber \\
&& \langle  Ax,Ax\rangle \leq \langle x,x\rangle,\quad 4\langle Ax,Ax\rangle-4\langle Ax,x\rangle+\langle  x,x\rangle \leq \langle x,x\rangle \nonumber
\end{eqnarray}
that is
\begin{equation} \label{eq:char}
\Vert Ax \Vert^2 = \langle Ax,Ax\rangle \leq \langle Ax,x\rangle.
\end{equation}
(which would imply $\langle Ax,Ax\rangle\leq \langle x,x\rangle$, i.e.\ $\Vert Ax \Vert \leq \Vert x \Vert$, otherwise 
$\langle Ax,x\rangle$ would be $\leq \Vert Ax \Vert \Vert x \Vert <
 \Vert Ax \Vert^2$).   

\item 
Since
$\langle Ax,x\rangle=(\cos\angle Ax,x)\cdot\Vert Ax\Vert \Vert x\Vert$
 we may write (\ref{eq:char}) as
\begin{equation} \label{eq:char1}
\Vert Ax \Vert \leq (\cos\angle Ax,x)\Vert x\Vert.
\end{equation}

\item
Anyhow, we find that $A$ FNE implies $\langle Ax,x\rangle$ always $\geq 0$ i.e.\ $A$ is NE and \textit{monotone}.

\item
One may view (\ref{eq:char1}) as a \textit{``balancing effect''} -- if the angle of rotation effected by $P$ acting on $x$ becomes close to $90^{\circ}$, then the norm will be reduced considerably. That over and above the latter angle being $\le 90^{\circ}$ as $A$ is monotone, and the norm of $Px$ being $\le$ that of $x$, as $A$ is NE.

\end{itemize}

The requirement for a NE operator $A$ on a Hilbert space $\mathcal{H}$ to be FNE, is $\Vert 2A-\mathbf{1}\Vert \leq 1$, i.e.\ $\Vert A-\mathbf{1/2}\Vert \le 1/2$, in other words $\Vert A\Vert_{m}(1/2)\le 1/2$,\, that is $\Vert A\Vert_{am}(1/2)\le 1$, and thus boils down to
\begin{theorem} 
For a NE operator $A$ to be FNE it is necessary and sufficient that
$\operatorname{Hor}(A)\ge 1/2$.

A NE Hermitian operator is FNE if and only if it is nonnegative semidefinite,
\end{theorem}
\begin{proof}
For the second assertion see Thm.\ \ref{th:Her}.
\end{proof}

\end{document}